\documentclass[11pt]{amsart}
\usepackage{amssymb,amsmath}

\theoremstyle{plain}
\newtheorem{thrm}{Theorem}[section]
\newtheorem{lemma}[thrm]{Lemma}
\newtheorem{prop}[thrm]{Proposition}

\newtheorem{rmrk}[thrm]{Remark}
\newtheorem{dfn}[thrm]{Definition}

\numberwithin{equation}{section}

\begin{document}
\newcommand{\SL}{\mathcal L^{1,p}( D)}
\newcommand{\Lp}{L^p( Dega)}
\newcommand{\CO}{C^\infty_0( \Omega)}
\newcommand{\Rn}{\mathbb R^n}
\newcommand{\Rm}{\mathbb R^m}
\newcommand{\R}{\mathbb R}
\newcommand{\Om}{\Omega}
\newcommand{\Hn}{\mathbb H^n}
\newcommand{\aB}{\alpha B}
\newcommand{\eps}{\epsilon}
\newcommand{\BVX}{BV_X(\Omega)}
\newcommand{\p}{\partial}
\newcommand{\IO}{\int_\Omega}
\newcommand{\bG}{\boldsymbol{G}}
\newcommand{\bg}{\mathfrak g}
\newcommand{\Bux}{\mbox{Box}}
\newcommand{\al}{\alpha}
\newcommand{\til}{\tilde}


\title[Interior Cauchy-Schauder estimates for the heat flow, etc.]{Interior Cauchy-Schauder estimates for the heat flow in Carnot-Carath\'eodory spaces}

\dedicatory{Dedicated to Neil Trudinger, with friendship and deep
admiration, on the occasion of his 65th birthday}

\author{Donatella Danielli}
\address{Department of Mathematics\\ Purdue University \\ West Lafayette, IN 47907}
\email[Donatella Danielli]{danielli@math.purdue.edu}
\thanks{First author supported in part by NSF CAREER Grant, DMS-0239771}

\author{Nicola Garofalo}
\address{Department of Mathematics\\Purdue University \\
West Lafayette, IN 47907}
\email[Nicola Garofalo]{garofalo@math.purdue.edu}
\address{Dipartimento di Metodi e Modelli Matematici per le Scienze Applicate\\ Universit\`a di Padova\\ 35131 Padova, Italy}
\email[Nicola Garofalo]{garofalo@dmsa.unipd.it}
\thanks{Second author supported in part by NSF Grant DMS-0701001}

%

%
%
\date{\today}

\maketitle


\baselineskip 16pt

\section{\textbf{Introduction.}}

\vskip 0.2in

The purpose of this paper is to establish some basic interior
estimates of Cauchy-Schauder type for the heat flow associated with
a system $X = \{X_1,...,X_m\}$ of $C^\infty$ vector fields in $\Rn$
satisfying H\"ormander's finite rank condition \cite{H}
\begin{equation}\label{frc}
rank\ Lie[X_1,...,X_m]\ \equiv\ n\ .
\end{equation}

Besides from playing an important role in the applications, such
estimates also have an obvious independent interest. Similarly to
the classical ones for elliptic equations, our results are tailored
on the intrinsic geometry of the system $X = \{X_1,...,X_m\}$. We
consider the \emph{sub-Laplacian} associated with $X$
\begin{equation}\label{SL}
\mathcal L\ =\ -\ \sum_{j=1}^m X_j^* X_j, ,
\end{equation}
and the corresponding \emph{heat operator} in $\mathbb R^{n+1}$
\begin{equation}\label{HO}
\mathcal H\ =\ \mathcal L\ -\ \frac{\partial}{\partial t}\ .
\end{equation}

Thanks to H\"ormander's theorem \cite{H}, the assumption \eqref{frc}
guarantees the hypoellipticity of both $\mathcal L$ and $\mathcal H$
on their respective domains. Therefore, distributional solutions of
either $\mathcal L u = f$, or $\mathcal H u = F$ are $C^\infty$,
wherever such are $f$ or $F$. Using the basic results in \cite{RS},
the existence of a fundamental solution for $\mathcal L$ and its
size estimates were established independently by Sanchez-Calle
\cite{SC} and by Nagel, Stein and Wainger \cite{NSW}. Corresponding
Gaussian estimates for the heat kernels were independently obtained
by Jerison and Sanchez-Calle \cite{JSC} and by Kusuoka and Stroock
\cite{KS1}, \cite{KS2}. One should also see \cite{VSC}.

To state our main result we define for $z = (x,t)\in \mathbb
R^{n+1}$ and $r>0$ the parabolic cylinders
\[
Q(z,r)\ =\ B(x,r)\ \times\ (t - r^2, t)\ ,\quad\quad\quad Q^-(z,r)\
=\ B(x,r)\ \times\ (t - r^2, t - \frac{r^2}{4})\ ,
\]
where we have indicated by $B(x,r) = \{y \in \Rn \mid d(x,y) < r\}$
the metric ball in the Carnot-Carath\'eodory (CC) distance
associated with $X$. Slightly abusing the notation we will write
$|B(x,r)|$ for the $n$-dimensional Lebesgue measure of the ball
$B(x,r)$ in $\Rn$, and also $|Q(z,r)|$ for the $(n+1)$-dimensional
Lebesgue measure of the cylinder $Q(z,r)$ in $\mathbb R^{n+1}$. We
note explicitly that $|Q(z,r)| = r^2\ |B(x,r)|$. Here is the main
result in this paper.

\medskip

\begin{thrm}\label{T:PS}
Let $D\subset \mathbb R^{n+1}$ be an open set and suppose that $u$
solves $\mathcal H u = 0$ in $D$. There exists $R_o>0$, depending on
$D$ and $X$, such that for every $z_o\in D$ and $0<r\leq R_o$ for
which $\overline{Q}(z_o,r)\subset D$, one has for any $s, k \in
\mathbb N \cup \{0\}$
\[
\underset{{Q}(z_o,r/2)}{\sup}\ \bigg|\frac{\partial^k}{\partial t^k}
X_{j_1}X_{j_2}...X_{j_s}\  u\bigg|\ \leq\ \frac{C}{r^{s + 2k}}\
\frac{1}{|Q(z_o,2r)|}\ \int_{Q(z_o,2r)}\ |u|\ d\xi\ d\tau\ ,
\]
for some constant $C=C(D,X,s,k)>0$. In the above estimate, for every
$i = 1,...,s,$ the index $j_i$ runs in the set $\{1,...,m\}$. When
$u\geq 0$ in $D$, then the above estimate can be replaced by
\[
\underset{{Q^-}(z_o,r/2)}{\sup}\ \bigg|\frac{\partial^k}{\partial t^k} X_{j_1}X_{j_2}...X_{j_s}\  u\bigg|\ \leq\ \frac{C}{r^{s + 2k}}\ u(z_o)\ .
\]
\end{thrm}

Theorem \ref{T:PS} generalizes (and contains) the following
steady-state Cauchy-Schauder type estimates which were found in
\cite{CGN3}.

\begin{thrm}\label{T:Harmonic}
Let $\Om\subset \Rn$ be a bounded open set and suppose that
$\mathcal L u = 0$ in $\Om$. There exists $R_o>0$, depending on
$\Om$ and $X$, such that for every $x\in \Om$ and $0<r\leq R_o$ for
which $\overline{B}(x,r)\subset \Om$, one has for any $s\in \mathbb
N$
\[
|X_{j_1}X_{j_2}...X_{j_s}u(x)|\ \leq\ \frac{C}{r^s}\
\underset{\overline{B}(x,r)}{max}\ |u|,
\]
for some constant $C=C(\Om,X,s)>0$. In the above estimate, for every
$i = 1,...,s,$ the index $j_i$ runs in the set $\{1,...,m\}$ .
\end{thrm}

This result was proved in \cite{CGN3} with a different approach
which exploited the geometry of the level sets of the fundamental
solution of the relevant sub-Laplacian. That approach however does
not seem to work in the time-dependent setting of this paper since
one presently lacks some very delicate asymptotic estimates of the
relevant heat kernels. In the special setting of the Heisenberg
group $\Hn$ such estimates were obtained in \cite{GS} using a
refined asymptotic analysis of the Fourier integrals involved in
Gaveau's explicit fundamental solution for the heat equation, see
\cite{Ga}. But the explicit calculations in \cite{GS} are out of
question in the general setting of this paper. We have been able to
get around such lack of estimates by: (i) working with intrinsic
parabolic cylinders in which the base is not a Carnot-Carath\'eodory
ball, but rather a level set of the steady-state fundamental
solution. This allows us to construct some appropriate $C^\infty$
cut-off functions, see Lemma \ref{L:parcutoff}; (ii) mimicking the
basic idea in E. E. Levi's method of the parametrix.

To see that Theorem \ref{T:PS} contains Theorem \ref{T:Harmonic} it
suffices to apply the former result to the function $u(x,t) = u(x)$,
where $u(x)$ which solves $\mathcal L u = 0$ in $\Om\subset \Rn$. It
is worth emphasizing that, in contrast with the classical case, in
the subelliptic setting any derivative $X_{j_i} u$ of a solution to
$\mathcal H u = 0$ fails to be itself a solution of $\mathcal H u =
0$. Our proof of Theorem \ref{T:PS} is based on the Gaussian
estimates of the heat kernel found in \cite{JSC}, \cite{KS1},
\cite{KS2}. In particular, we will assume such estimates in the
version which was obtained in \cite{KS1}. With such basic tool in
hands, we will then use arguments which are reminiscent of the
classical ones originated with E. E. Levi's method of the
parametrix, see \cite{Fr}, or also \cite{E}. It is worth mentioning
at this point that, in the special setting of Carnot groups, our
arguments combined with the non-isotropic group dilations, provide a
simplified approach to several of the local and global results
available in the literature.

Similarly to its classical ancestor, Theorem \ref{T:PS} has many
basic consequences. For instance, one can use it to study the
regularity at the boundary for solutions to the equation $\mathcal H
u = 0$ both in cylindrical and non-cylindrical domains in $\mathbb
R^{n+1}$. Basic applications of the Cauchy-Schauder estimates in the
time-independent subelliptic Dirichlet can be found in the works
\cite{CGN1}, \cite{LU}, \cite{CGN2}, \cite{LU2}, \cite{CGN3}. In a
different direction, such estimates were used in the works
\cite{CDG1}, \cite{CDG2}, \cite{CDG3}. Theorem \ref{T:PS} can also
be used to deduce some interesting Cauchy-Liouville type properties
of solutions to parabolic equations such as \eqref{HO}.

\vskip 0.6in

\section{\textbf{Preliminaries}}\label{S:Prelim}

\vskip 0.2in

In $\Rn$, with $n\geq 3$, we consider a system  $X =
\{X_1,...,X_m\}$ of $C^\infty$ vector fields satisfying H\"ormander
finite rank condition. A piecewise $C^1$ curve $\gamma:[0,T]\to \Rn$
is called \emph{sub-unitary} if whenever $\gamma'(t)$ exists one has
for every $\xi\in\Rn$
\[
<\gamma'(t),\xi>^2\ \leq\ \sum_{j=1}^m <X_j(\gamma(t)),\xi>^2 .
\]

We note explicitly that the above inequality forces $\gamma '(t)$ to
belong to the span of $\{X_1(\gamma (t)),...,$ $ X_m(\gamma (t))\}$.
The sub-unit length of $\gamma$ is by definition $l_s(\gamma)=T$.
Given $x, y\in \Rn$, denote by $\mathcal S_\Om(x,y)$ the collection
of all sub-unitary $\gamma:[0,T]\to \Om$ which join $x$ to $y$. The
Chow-Rashevsky accessibility theorem  states that, given a connected
open set $\Om\subset \Rn$, for every $x,y\in \Om$ there exists
$\gamma \in \mathcal S_\Om(x,y)$, see \cite{Chow}, \cite{Ra}. As a
consequence, if we pose
\[
d_{\Om}(x,y)\ =\ \text{inf}\ \{l_s(\gamma)\mid \gamma \in \mathcal S_\Om(x,y)\} ,
\]
we obtain a distance on $\Om$, called the \emph{Carnot-Carath\'eodory distance on $\Om$}, associated with the system $X$. When $\Om = \Rn$, we write $d(x,y)$ instead of $d_{\Rn}(x,y)$. It is clear that $d(x,y) \leq d_\Om(x,y)$, $x, y\in \Om$, for every connected open set $\Om \subset \Rn$. In \cite{NSW} it was proved that for every connected $\Om \subset \subset \Rn$ there exist $C, \epsilon >0$ such that
\begin{equation}\label{CCeucl}
C\ |x - y|\ \leq d_\Om(x,y)\ \leq C^{-1}\ |x - y|^\epsilon ,
\quad\quad\quad x, y \in \Om ,
\end{equation}
see also \cite{Be}. This gives $d(x,y)\ \leq C^{-1} |x -
y|^\epsilon$, $x, y\in \Om$, and therefore
\[
i: (\Rn, |\cdot|)\to (\Rn, d) \quad\quad\quad is\,\ continuous .
\]

It is easy to see that also the continuity of the opposite inclusion holds \cite{GN1}, hence the metric and the Euclidean topology are compatible.

For $x\in \Rn$ and $r>0$, we let $B_d(x,r)\ =\ \{y\in \Rn\mid d(x,y) < r \}$.
The basic properties of these balls were established by Nagel, Stein and Wainger in their seminal paper \cite{NSW}. Denote by $Y_1,...,Y_l$ the collection of the $X_j$'s and of those commutators which are needed to generate $\Rn$. A formal ``degree" is assigned to each $Y_i$, namely the corresponding order of the commutator. If $I = (i_1,...,i_n), 1\leq i_j\leq l$ is a $n$-tuple of integers, following \cite{NSW} we let $d(I) = \sum_{j=1}^n deg(Y_{i_j})$, and $a_I(x) = \text{det}\ (Y_{i_1},...,Y_{i_n})$. The \emph{Nagel-Stein-Wainger polynomial} is defined by
\begin{equation}\label{pol}
\Lambda(x,r)\ =\ \sum_I\ |a_I(x)|\ r^{d(I)}, \quad\quad\quad\quad r > 0.
\end{equation}

For a given bounded open set $U\subset \Rn$, we let
\begin{equation}\label{Q}
Q\ =\ \text{sup}\ \{d(I)\mid \ |a_I(x)| \ne 0, x\in U\},\quad\quad Q(x)\ =\ \text{inf}\ \{d(I)\mid |a_I(x)| \ne 0\} ,
\end{equation}
and notice that $n\leq Q(x)\leq Q$. The numbers $Q$ and $Q(x)$ are respectively called the \emph{local homogeneous dimension} of $U$ and the homogeneous dimension at $x$ with respect to the system $X$.

\medskip
\begin{thrm}[\textbf{\cite{NSW}}]\label{T:db}
For every bounded set $U\subset\Rn$, there exist
constants $C, R_o>0$ such that, for any $x\in U$, and $0 < r \leq R_o$,
\begin{equation}\label{nsw2}
C\ \Lambda(x,r)\ \leq\ |B_d(x,r)|\ \leq\  C^{-1}\ \Lambda(x,r).
\end{equation}
As a consequence, one has with $C_1 = 2^Q$
\begin{equation}\label{dc}
|B_d(x,2r)|\ \leq\ C_1\ |B_d(x,r)| \qquad\text{for every}\quad x\in U\quad\text{and}\quad 0< r \leq R_o.
\end{equation}
\end{thrm}

\medskip

The numbers $C_1, R_o$ in \eqref{dc} will be referred to as the
\emph{characteristic local parameters} of $U$. We notice that
\eqref{dc} implies for every $x\in U$, $0<r\leq R_o$ and $0\leq
t\leq 1$
\begin{equation}\label{dc2}
|B(x,tr)|\ \geq\ C_1 t^Q |B(x,r)|\ . \end{equation}

Because of (2.2), if we let
\begin{equation}\label{E}
E(x,r)\ =\ \frac{\Lambda(x,r)}{r^{2}} ,
\end{equation}
then the function $r \to E(x,r)$ is strictly increasing. We denote by $F(x,\cdot)$ the  inverse function of $E(x,\cdot)$, so that $F(x,E(x,r)) = r$. Let $\Gamma(x,y)=\Gamma(y,x)$ be the positive fundamental solution of the sub-Laplacian
\[
\mathcal L\ =\ \sum_{j=1}^m X_j^*X_j ,
\]
and consider its level sets
\[
\Om(x,r)\ =\ \left\{y\in \Rn \mid \Gamma(x,y) > \frac{1}{r}\right\} .
\]

The following definition plays a key role in this paper.

\medskip

\begin{dfn}\label{D:lballs}
For every $x\in \Rn$, and $r>0$, the set
\[
B_X(x,r)\ =\ \left\{y\in \Rn \mid \Gamma(x,y) >
\frac{1}{E(x,r)}\right\}
\]
will be called the $X$-\emph{ball}, centered at $x$ with radius $r$.
\end{dfn}

\medskip

We note explicitly that
\[
B_X(x,r)\ =\ \Om(x,E(x,r)),\quad \quad \text{and that} \quad \quad
\Om(x,r)\ =\ B_X(x,F(x,r)).
\]

One of the main geometric properties of the $X$-balls, is that they
are equivalent to the Carnot-Carath\'eodory balls. To see this, we
recall the following  important result, established independently in
\cite{NSW}, \cite{SC}, see also \cite{FSC}. Hereafter, the notation
$Xu=(X_1u,...,X_mu)$ indicates the sub-gradient of a function $u$,
whereas $|Xu|=(\sum_{j=1}^m(X_ju)^2)^\frac{1}{2}$ will denote its
length.

\medskip

\begin{thrm}\label{T:NSW}
 Given a bounded set $U\subset \Rn$, there exists $R_o$, depending on $U$ and on $X$, such that for $x\in U,\ 0<d(x,y)\leq R_o$, one has for $s\in \mathbb{N}\cup\{0\}$, and for some constant $C=C(U, X, s) >0$
\begin{align}\label{gradgamma}
& |X_{j_1}X_{j_2}...X_{j_s}\Gamma(x,y)|\ \leq\ C^{-1}\ \frac{d(x,y)^{2-s}}{|B_d(x,d(x,y))|},
\\
& \Gamma(x,y)\ \geq \ C\ \frac{d(x,y)^2}{|B_d(x,d(x,y))|} .
\notag
\end{align}
In the first inequality in \eqref{gradgamma}, one has $j_i\in \{1,...,m\}$ for $i=1,...,s$, and $X_{j_i}$ is allowed to act on either $x$ or $y$.
\end{thrm}

\medskip

In view of \eqref{dc}, \eqref{gradgamma}, it is now easy to recognize that, given a bounded set $U\subset \Rn$, there exists $a>1$, depending on $U$ and $X$, such that
\begin{equation}\label{equiv}
B(x,a^{-1}r)\ \subset\  B_X(x,r)\ \subset\ B(x,ar),
\end{equation}
for $x\in U, 0<r\leq R_o$. We observe that, as a consequence of \eqref{nsw2}, and of \eqref{gradgamma}, one has
\begin{equation}\label{F}
C\ d(x,y)\ \leq\  F\left(x,\frac{1}{\Gamma(x,y)}\right)\ \leq\  C^{-1}\ d(x,y),
\end{equation}
for all $x\in U, 0<d(x,y)\leq R_o$.

We observe that for a Carnot group $\bG$ of step $k$, if
$\bg=V_1\oplus ...\oplus V_k$ is a stratification of the Lie algebra
of $\bG$, then one has $\Lambda(x,r)=const\ r^Q$, for every $x\in
\bG$ and every $r>0$, with $Q=\sum_{j=1}^k\;j\;dimV_j$, the
homogeneous dimension of the group $\bG$. In this case  $Q(x) \equiv
Q$, see \cite{F}, \cite{BLU}, \cite{G}.

It is important to keep in mind the following basic properties of a
Carnot-Carath\'eodory space.

\medskip

\begin{prop}\label{P:compact}
$(\Rn,d)$ is locally compact.  Furthermore, for any
bounded set $U\subset\Rn$ there exists $R_o=R_o(U)>0$ such that the closed
balls $\bar B(x_o,R)$, with $x_o\in U$ and $0<R<R_o$, are compact.
\end{prop}

\medskip

\begin{rmrk}
Compactness of balls of large radii may fail in general, see
\cite{GN1}.  However, there are important cases in which
Proposition~\ref{P:compact} holds globally, in the sense that one
can take $U$ to coincide with the whole ambient space and
$R_o=\infty$. One example is that of Carnot groups. Another
interesting case is that when the vector fields $X_j$ have
coefficients which are globally Lipschitz, see \cite{GN1},
\cite{GN2}. Henceforth, for any given bounded set $U\subset\Rn$ we
will always assume that the local parameter $R_o$ has been chosen so
to accommodate Proposition \ref{P:compact}.
\end{rmrk}

\medskip

Keeping in mind the applicability of the results in this paper to
the study of the Dirichlet problem, the basic reference on this
subject for operators of H\"ormander type remains the pioneer study
of Bony \cite{B}. For more recent developments the reader should
consult \cite{CG}, \cite{CGN3} and the references therein.

The following basic result was established in \cite{KS1},
\cite{KS2}, see also \cite{JSC}.

\medskip

\begin{thrm}\label{T:KS}
The fundamental solution $p(x,t;\xi,\tau) = p(x,\xi;t-\tau)$ with singularity at $(\xi,\tau)$ satisfies the following size estimates : there exists $M = M(X)>0$ and for every
$k , s\in \mathbb{N}\cup\{0\}$, there exists a constant $C=C(X, k, s) >0$, such that
\begin{equation}\label{gaussian1}
\bigg|\frac{\partial^k}{\partial t^k} X_{j_1}X_{j_2}...X_{j_s}p(x,t;\xi,\tau)\bigg|\ \leq\ \frac{C}{(t - \tau)^{s + 2k}}\ \frac{1}{|B(x,\sqrt{t - \tau})|}\ \exp\ \bigg( -\ \frac{M d(x,\xi)^2}{t - \tau}\bigg)\ ,
\end{equation}
\begin{equation}\label{gaussian2}
p(x,t;\xi,\tau)\ \geq \ \frac{C^{-1}}{|B(x,\sqrt{t - \tau})|}\ \exp\ \bigg( -\ \frac{M^{-1} d(x,\xi)^2}{t - \tau}\bigg)\ ,
\end{equation}
for every $x,\xi \in \Rn$, and any $-\infty < \tau < t < \infty $.
\end{thrm}

\vskip 0.6in


\section{\textbf{Proof of Theorem \ref{T:PS}}}\label{S:proof}

\vskip 0.2in

In the sequel we fix a point $z_o = (x_o,t_o) \in \mathbb R^{n+1}$ and consider the following parabolic cylinders
\begin{equation}\label{Q}
Q(z_o,r)\ =\ B(x_o,r)\ \times\ (t_o - r^2, t_o)\ ,
\end{equation}
\begin{equation}\label{Qminus}
Q^-(z_o,r)\ =\ B(x_o,r)\ \times\ (t_o - r^2, t_o - \frac{r^2}{4}) .
\end{equation}

We will also need cylinders based on the $X$-balls
\begin{equation}\label{XQ}
Q_X(z_o,r)\ =\ B_X(x_o,r)\ \times\ (t_o - r^2, t_o)\ ,
\end{equation}
\begin{equation}\label{XQminus}
Q^-_X(z_o,r)\ =\ B_X(x_o,r)\ \times\ (t_o - r^2, t_o - \frac{r^2}{4}) .
\end{equation}

The following basic lemma, which constitutes a generalization of a
result obtained in \cite{CGL} for the case $s = 1$, is the main
motivation for introducing the cylinders $Q_X(z_o,r)$. If we work
with the $CC$ balls $B(x_o,r)$ then the existence of a smooth
cut-off function fails since in general only the first derivatives
of the $CC$ distance with respect to the vector fields $X_1,...,X_m$
are bounded, see e.g. \cite{GN2}.

\medskip

\begin{lemma}\label{L:cutoff}
For every $x_o\in \Rn$ and $r>0$ there exists a function $\chi \in C^\infty_o(B_X(x_o,2r))$, $0 \leq \chi \leq 1$, such that $\chi \equiv 1$ on $B_X(x_o,r)$, and moreover for each $s\in \mathbb N \cup \{0\}$ there exists $C = C(X,s) > 0$ such that
\[
|X_{j_1}X_{j_2}...X_{j_s} \chi|\ \leq\ \frac{C}{r^s}\ .
\]
\end{lemma}

\begin{proof}[\textbf{Proof}]
Let $f\in C^\infty_o([0,\infty))$, $0 \leq f \leq 1$, $f\equiv 1$ on $[0,r]$, $supp\ f \subset [0,2r)$, and such that for every $s\in \mathbb N$ one has
\begin{equation}\label{stimef}
|f^{(s)}(\sigma)|\ \leq\ \frac{C(s)}{r^s}\quad\quad\quad\text{for every}\quad \sigma \in [r,2r]\ .
\end{equation}

We define $\chi(x) = f(\rho_{x_o}(x))$. This function clearly possesses the desired support properties. The estimates for the derivatives of $\chi$ along the vector fields $X_1,...,X_m$ now follow by recurrence from \eqref{stimef} and from the following estimates for the regularized distance $\rho_{x_o}$
\begin{equation}\label{Xregdis}
|X_{j_1}X_{j_2}...X_{j_s} \rho_{x_o}(x)|\ \leq\ \frac{C(s)}{\rho_{x_o}(x)^{s-1}}\ \quad\quad\quad x \not= x_o\ .
\end{equation}

We leave the details to the reader.

\end{proof}

\medskip

Using Lemma \ref{L:cutoff} we now obtain a similar ad hoc result on the parabolic cylinders $Q_X(z_o,r)$.

\medskip

\begin{lemma}\label{L:parcutoff}
For every $z_o = (x_o,t_o)\in \mathbb R^{n+1}$ and $r>0$ there exists a function
\[
\zeta\ \in\ C^\infty_o\left(B_X(x_o,2r)\ \times\ (t_o - 4 r^2, t_o]\right)\ ,
\]
 $0 \leq \zeta \leq 1$, such that $\zeta \equiv 1$ on $\overline Q_X(z_o,r)$, and moreover for each $s , k \in \mathbb N \cup \{0\}$ there exists $C = C(X,s,k) > 0$ such that
\begin{equation}\label{estzeta}
|\frac{\partial^k}{\partial t^k} X_{j_1}X_{j_2}...X_{j_s} \zeta|\ \leq\ \frac{C}{r^{s+2k}}\ .
\end{equation}
\end{lemma}

\begin{proof}[\textbf{Proof}]
We choose a function $h\in C^\infty(\mathbb R)$, such that $0 \leq h \leq 1$, $h \equiv 1$ on $[t_o - r^2, \infty)$, $supp\ h \subset (t_o - 4 r^2, \infty)$, and for which one has for every $k\in \mathbb N$
\[
|h^{(k)}(t)|\ \leq\ \frac{C(k)}{r^{2k}}\ \quad\quad\quad\text{for every}\quad t \in [t_o - 4 r^2, t_o - r^2]\ .
\]

We define
\[
\zeta(z)\ =\ \zeta(x,t)\ =\ \chi(x)\ h(t)\ ,
\]
where $\chi$ is the function in Lemma \ref{L:cutoff}. From the definitions of $\chi$ and $h$ it is clear that $\zeta \equiv 1$ on $\overline Q_X(z_o,r)$. Since
\[
|\frac{\partial^k}{\partial t^k} X_{j_1}X_{j_2}...X_{j_s} \zeta(z)|\ =\ |X_{j_1}X_{j_2}...X_{j_s} \chi(x)|\ |h^{(k)}(t)|\ ,
\]
the estimate \eqref{estzeta} follows from the corresponding ones for $\chi$ and $h$.

\end{proof}

\medskip

We will need the following form of Duhamel's principle.

\medskip

\begin{prop}\label{P:Duhamel}
Let $F\in C^\infty_o(\mathbb R^{n} \times [0,\infty))$, and define
\begin{equation}\label{Duhamel}
w(z)\ =\ \int_0^t\ \int_{\Rn}\ p(x,\xi;t - \tau)\ F(\xi,\tau)\ d\xi\ d\tau\ ,
\end{equation}
 where $p(x,\xi;t - \tau)$ denotes the positive fundamental solution of $\mathcal H$ with singularity at $(\xi,\tau)$. One has $w\in C^\infty(\Rn \times (0,\infty))$, and moreover
\begin{equation}\label{IVP}
\mathcal H w\ =\ -\ F\ \quad\quad\quad\text{in}\quad \Rn \times (0,t_o)\ ,\quad\quad w(x,0)\ =\ 0 \ ,\quad x \in \Rn\ .
\end{equation}
\end{prop}

\begin{proof}[\textbf{Proof}]
Follows along classical lines.

\end{proof}

In what follows, given an open set $D\subset \R^{n+1}$, we indicate
with $\Gamma^{2,1}(D)$ the collection of all continuous functions on
$D$ possessing two continuous derivatives with respect to the vector
fields $X_1,...,X_m$, and one continuous derivative with respect to
the variable $t$.

\begin{thrm}\label{T:uniqueness1}
There exists at most one solution $u\in \Gamma^{2,1}(\Rn \times (0,T_1])$ to the differential inequality
\begin{equation}\label{geq}
u(x,t)\ \mathcal H u(x,t)\ \geq\ 0\ \quad\quad\quad (x,t) \in \Rn \times (0,T_1)\ ,
\end{equation}
such that $u(x,0) = 0$ for $x\in \Rn$, and for which
\begin{equation}\label{unique1}
|u(x,t)|\ p(x_o,x;1)\ \leq\ A\ ,\quad\quad\quad\quad x\in \Rn, 0 < t < T_1\ ,
\end{equation}
for some $A > 0$, and some $x_o\in \Rn$. In particular, given $\phi \in C(\Rn)$, $F\in C(\Rn \times (0,T_1))$, there exists a unique solution to the Cauchy problem
\begin{equation}\label{IVP2}
\mathcal H u\ =\ F\ \quad\quad\quad\text{in}\quad \Rn \times (0,T_1)\ ,\quad\quad u(x,0)\ =\ \phi(x) \ ,\quad x \in \Rn\ ,
\end{equation}
satisfying the constraint \eqref{unique1}.
\end{thrm}

\begin{proof}[\textbf{Proof}]
We introduce the function
\begin{equation}\label{IVP3}
\phi(R)\ \overset{def}{=}\ \int_{\Rn}\ u^2(\xi,R^2)\ p(x,\xi;T - R^2)\ d\xi\ ,\quad\quad\quad 0 \leq R < \sqrt T\ ,
\end{equation}
where $x\in \Rn$ is fixed and the number $T<T_1$ will be chosen suitably small in a moment. We have from \eqref{unique1} and the Gaussian estimates \eqref{gaussian1}, \eqref{gaussian2} in Theorem \ref{T:KS}
\begin{align}\label{IVP4}
& u^2(\xi,R^2)\ p(x,\xi;T - R^2)\ \leq\ A^{-2}\ p(x_o,\xi;1)^{-2}\ p(x,\xi;T - R^2)
\\
& \leq\  \frac{C^3}{A^2}\ \frac{|B(x_o,1)|^2}{|B(x,\sqrt{T-R^2})|}\ \exp\ \bigg\{\frac{2  d(x_o,\xi)^2}{M}\bigg\} \exp\ \bigg\{ -  \frac{M d(x,\xi)^2}{T - R^2}\bigg\}\ .
\notag
\end{align}

At this point we choose
\begin{equation}\label{IVP5}
T_o\ <\ \min\ \left\{T_1,\frac{M^2}{4}\right\}\ ,
\end{equation}
and let $0<T<T_o$. Since the triangle inequality gives
\[
d(x_o,\xi)^2\ \leq\ d(x_o,x)^2\ +\ d(x,\xi)^2\ +\ 2\ d(x_o,x)\ d(x,\xi)\ ,
\]
we obtain from \eqref{IVP4}
\begin{align}\label{IVP6}
& u^2(\xi,R^2)\ p(x,\xi;T - R^2)\ \leq\ \frac{C^3}{A^2} \ \frac{|B(x_o,1)|^2}{|B(x,\sqrt{T-R^2})|}\ \exp \bigg\{\frac{2 d(x_o,x)^2}{M}\bigg\}
\\
& \times \exp \bigg\{\frac{2 d(x,\xi)^2}{M} + 4 M d(x_o,x) d(x,\xi)\bigg\}\exp \bigg\{ -  \frac{M d(x,\xi)^2}{T}\bigg\}\ .
\notag\\
& \leq\ \frac{C^3}{A^2}\ \frac{|B(x_o,1)|^2}{|B(x,\sqrt{T-R^2})|}\ \exp \bigg\{\frac{2 d(x_o,x)^2}{M}\bigg\}
\notag\\
& \times \exp \bigg\{\frac{2 d(x,\xi)^2}{M} + 4 M d(x_o,x) d(x,\xi)\bigg\}\exp \bigg\{ -  \frac{4 d(x,\xi)^2}{M}\bigg\}\ ,
\notag
\end{align}
where in the last inequality we have used \eqref{IVP6}. From \eqref{IVP6} it is clear that if $d(x,\xi) \geq 4 d(x_o,x)$, then
\begin{align*}
& u^2(\xi,R^2)\ p(x,\xi;T - R^2)\  \leq\ \frac{C^3}{A^2}\ \frac{|B(x_o,1)|^2}{|B(x,\sqrt{T-R^2})|}\ \exp \bigg\{\frac{2 d(x_o,x)^2}{M}\bigg\} \exp \bigg\{ -  M d(x,\xi)^2 \bigg\}\ .
\end{align*}

From this estimate we infer that the integral in \eqref{IVP3} defining $\phi(R)$ is absolutely convergent for any $x\in \Rn$ and every $0<T<T_o$, with $T_o$ satisfying \eqref{IVP5}.

Differentiating under the integral sign we find
\begin{align}\label{IVP7}
\phi'(R)\ &=\ 4\ R\ \int_{\Rn}\ u(\xi,R^2)\ \partial_t u(\xi,R^2)\ p(x,\xi;T - R^2)\ d\xi
\\
& -\ 2\ R\ \int_{\Rn}\ u^2(\xi,R^2)\ \partial_t p(x,\xi;T - R^2)\ d\xi\ .
\notag
\end{align}

We now use the self-adjointness of the sub-Laplacian $\mathcal L$, and the equation satisfied by $p$, which gives for every $x, \xi\in \Rn$ and every $R\in (0,\sqrt T)$
\[
\partial_t p(x,\xi;T - R^2)\ =\ \mathcal L_\xi p(x,\xi;T - R^2)\ .
\]

Substituting the latter equation in the right-hand side of \eqref{IVP7} we obtain
\begin{align}\label{IVP8}
\phi'(R)\ &=\ 4\ R\ \int_{\Rn}\ u(\xi,R^2)\ \partial_t u(\xi,R^2)\ p(x,\xi;T - R^2)\ d\xi
\\
& -\ 2\ R\ \int_{\Rn}\ u^2(\xi,R^2)\  \mathcal L_\xi p(x,\xi;T - R^2)\ d\xi\ .
\notag
\end{align}

We now integrate by parts in the second integral in the right-hand side of \eqref{IVP8}. Using the assumption \eqref{unique1} we conclude that
\begin{align}\label{IVP9}
\phi'(R)\ &=\ 4\ R\ \int_{\Rn}\ u(\xi,R^2)\ \partial_t u(\xi,R^2)\ p(x,\xi;T - R^2)\ d\xi
\\
& -\ 2\ R\ \int_{\Rn}\ \mathcal L_\xi \big(u^2(\xi,R^2)\big)\  p(x,\xi;T - R^2)\ d\xi\ .
\notag
\end{align}

Observe that
\[
\mathcal L_\xi \big(u^2(\cdot,R^2)\big)(\xi)\ =\ 2\ u(\xi,R^2)\ \mathcal L u(\xi,R^2)\ +\ 2\ |X u(\xi,R^2)|^2\ .
\]
 Substitution in the latter equation gives
\begin{align}\label{IVP10}
\phi'(R)\ &=\ -\ 4\ R\ \int_{\Rn}\ u(\xi,R^2)\ \mathcal H u(\xi,R^2)\ p(x,\xi;T - R^2)\ d\xi
\\
& -\ 4\ R\ \int_{\Rn}\ |Xu(\xi,R^2)|^2\  p(x,\xi;T - R^2)\ d\xi\  .
\notag
\end{align}

If we now use the hypothesis \eqref{geq}, then we conclude from \eqref{IVP10}
\begin{equation}
\phi'(R)\ =\ -\ 4\ R\ \int_{\Rn}\ |Xu(\xi,R^2)|^2\  p(x,\xi;T - R^2)\ d\xi\ \leq \ 0\ ,
\end{equation}
and therefore the function $\phi(R)$ is decreasing. In particular, for every $0<R<\sqrt T$ one has
\[
\phi(R)\ \leq\ \phi(0)\ =\ \int_{\Rn}\ u(\xi,0)^2\ p(x,\xi;T)\ d\xi\ =\ 0\ .
\]

Letting $R\to \sqrt T$ in the latter inequality, by the $\delta$-function property of $p(x,\xi;\cdot)$ we infer
\[
\lim_{R\to \sqrt T}\ \phi(R)\ =\ u(x,T)^2\ =\ 0\ .
\]

By the arbitrariness of $x\in \Rn$, $T\in (0,T_o)$, we conclude that it must be $u \equiv 0$ in $\Rn \times (0,T_o)$. Repeating the above arguments on successive intervals $(T_o,2T_o)$, etc., we finally reach the conclusion $u \equiv 0$ in $\Rn \times (0,T_1)$.

To prove the second part of the theorem, we observe that by linearity it suffices to show that if $\mathcal H u = 0$ in $\Rn \times (0,T_1)$, $u(x,0) \equiv 0$, and $u$ satisfies the constraint \eqref{unique1}, then it must be $u \equiv 0$ in $\Rn \times (0,T_1)$. This is of course an immediate consequence of the first part.

\end{proof}

\medskip

We are now ready to establish the main result of this section.

\medskip

\begin{proof}[\textbf{Proof of Theorem \ref{T:PS}}]
Without loss of generality we assume that $t_o > 0$, and consider $r>0$ such that $t_o - 4 r^2 > 0$. For a given function $u\in C^\infty(\mathbb R^{n+1})$ we define
\begin{equation}\label{v}
v(z)\ \overset{def}{=}\ \zeta(z)\ u(z)\ ,
\end{equation}
where $\zeta$ is the cut-off whose existence has been established in Lemma \ref{L:parcutoff}.
One has
\begin{equation}\label{PS1}
\mathcal H v\ =\ u\ \mathcal H \zeta\ +\ \zeta\ \mathcal H u\ +\ 2\ <Xu,X\zeta>\ ,
\end{equation}
and also, thanks to the support properties of $\zeta$, we have
\begin{equation}\label{PS2}
v(x,0)\ =\ 0\ \quad\quad\quad\quad x\in \Rn\ .
\end{equation}

Suppose that $u$ be a solution to $\mathcal H u = 0$, then we obtain from \eqref{PS1}
\begin{equation}\label{PS3}
\mathcal H v\ =\ u\ \mathcal H \zeta\ +\ 2\ <Xu,X\zeta>\ \overset{def}{=}\ -\ F\ \quad\quad\quad\text{in}\quad \Rn \times (0,t_o)\ .
\end{equation}

We next define
\begin{equation}\label{PS4}
w(z)\ =\ \int_0^t\ \int_{\Rn}\ p(x,\xi;t - \tau)\ F(\xi,\tau)\ d\xi\ d\tau\ ,
\end{equation}
 where $p(x,\xi;t - \tau)$ is the positive fundamental solution of $\mathcal H$ with singularity at $(\xi,\tau)$. By Proposition \ref{P:Duhamel} we know that $w$ solves the problem
\begin{equation}\label{IVP}
\mathcal H w\ =\ -\ F\ \quad\quad\quad\text{in}\quad \Rn \times (0,t_o)\ ,\quad\quad w(x,0)\ =\ 0 \ ,\quad x \in \Rn\ .
\end{equation}

Comparing \eqref{PS2}, \eqref{PS3} and \eqref{IVP}, we see that $v$
and $w$ solve the same Cauchy problem in the strip $\Rn \times
(0,t_o)$. Furthermore, by the support properties of $\zeta$ we see
that $v\in L^\infty(\Rn \times (0,t_o))$. On the other hand, by the
same reasons we have $F\in L^\infty(\Rn \times (0,t_o))$, and the
definition of $w$ gives for $z = (x,t) \in \Rn \times (0,t_o))$
\[
|w(z)|\ \leq\ ||F||_{L^\infty(\Rn \times (0,t_o))}\ \int_0^t \int_{\Rn}\ p(x,\xi;t - \tau)\ d\xi\ d\tau\ \leq\ t_o\ ||F||_{L^\infty(\Rn \times (0,t_o))}\ ,
\]
since
\[
\int_{\Rn}\ p(x,\xi;t - \tau)\ d\xi\ =\ 1\ .
\]

Thanks to Theorem \ref{T:uniqueness1} we can thus conclude that $v \equiv w$ in $\Rn \times (0,t_o)$, therefore
\[
v(z)\ =\ \int_0^t\ \int_{\Rn} p(x,\xi;t - \tau)\ F(\xi,\tau)\ d\xi\ d\tau\ .
\]

Assume now that $z\in Q_X(z_o,r/2)$. Since $\zeta(z) = 1$, by the support properties of $p(x,\xi;t - \tau)$ and $\zeta$ we obtain
\begin{align}\label{PS5}
u(z)\ & =\ \int_{t_o - 4 r^2}^{t_o}\ \int_{B_X(x_o,2r)} p(x,\xi;t - \tau)\ F(\xi,\tau)\ d\xi\ d\tau\
\\
& =\  \int_{t_o - 4 r^2}^{t_o}\ \int_{B_X(x_o,2r)} p(x,\xi;t - \tau)\ \mathcal H \zeta(\xi,\tau)\ u(\xi,\tau)\ d\xi\ d\tau
\notag \\
& \ +\ 2\  \int_{t_o - 4 r^2}^{t_o}\ \int_{B_X(x_o,2r)} p(x,\xi;t - \tau)\ <X\zeta(\xi,\tau),Xu(\xi,\tau)> d\xi\ d\tau\ .
\notag
\end{align}

An integration by parts gives
\begin{align}\label{PS6}
&  2\  \int_{t_o - 4 r^2}^{t_o}\ \int_{B_X(x_o,2r)} p(x,\xi;t - \tau)\ <X\zeta(\xi,\tau),Xu(\xi,\tau)> d\xi\ d\tau\
\\
& =\  2\  \sum_{j=1}^m \int_{t_o - 4 r^2}^{t_o}\ \int_{\partial B_X(x_o,2r)} p(x,\xi;t - \tau)\ X_j \zeta(\xi,\tau)\ <X_j,\nu_\xi> u(\xi,\tau)\ d\xi\ d\tau\
\notag\\
& -\ 2\ \int_{t_o - 4 r^2}^{t_o}\ \int_{B_X(x_o,2r)} <Xp(x,\xi;t - \tau),X\zeta(\xi,\tau)> u(\xi,\tau)\ d\xi\ d\tau\
\notag\\
& -\ 2\ \sum_{j=1}^m \int_{t_o - 4 r^2}^{t_o}\ \int_{B_X(x_o,2r)}\ p(x,\xi;t - \tau)\ X_j X_j \zeta(\xi,\tau)\ u(\xi,\tau)\ d\xi\ d\tau\ ,
\notag
\end{align}
where in the first integral in the right-hand side of \eqref{PS6} we have denoted by $\nu_\xi$ the spacial component of the outer unit normal to the smooth manifold $\partial B_X(x_o,2r)$.  Since $X\zeta \equiv 0$ on $\partial B_X(x_o,2r)\times (t_o - 4 r^2,t_o)$, the above integral vanishes, and after substituting \eqref{PS6} into \eqref{PS5} we obtain for $z\in Q_X(z_o,r/2)$
\begin{equation}\label{PS7}
u(z)\ =\ \int_{Q_X(z_o,2r)}\ K(x,\xi;t - \tau)\ u(\xi,\tau)\ d\xi\ d\tau\ ,
\end{equation}
where
\begin{align}\label{PS8}
K(x,\xi;t - \tau)\ & =\ p(x,\xi;t - \tau)\ \mathcal H \zeta(\xi,\tau)\ +\ 2\ <Xp(x,\xi;t - \tau),X\zeta(\xi,\tau)>
\\
& +\ 2\  p(x,\xi;t - \tau)\ \sum_{j=1}^m X_j X_j \zeta(\xi,\tau)
\notag\\
& =\ p(x,\xi;t - \tau)\ \left\{\frac{\partial \zeta}{\partial \tau}\ +\ \sum_{j=1}^m X_j X_j \zeta\ -\ <\overset{\longrightarrow}{div X},X\zeta>\right\}(\xi,\tau)
 \notag\\
& +\ 2\ <Xp(x,\xi;t - \tau),X\zeta(\xi,\tau)>\ ,
\notag
\end{align}
where we have denoted $\overset{\longrightarrow}{div X} = (div X_1,...,div X_m)$. Recalling that $\zeta \equiv 1$ on $Q_X(z_o,r)$, and that $p(x,\xi;t - \tau) \equiv 0$ when $\tau \geq t$, we see that for every fixed $z = (x,t) \in Q_X(z_o,r/2)$, the integral in \eqref{PS7} is actually performed on the region
\begin{equation}\label{support}
\bigg[Q_X(z_o,2r)\ \setminus\ Q_X(z_o,r)\bigg]\ \cap\ \{(\xi,\tau)\in \mathbb R^{n+1} \mid \tau < t \}\ .
\end{equation}

This observation will be important in the sequel. First of all, we can differentiate the right-hand side of \eqref{PS7} under the integral sign to obtain
\begin{equation}\label{PS9}
\underset{Q_X(z_o,r/2)}{\sup}\ |\frac{\partial^k}{\partial t^k} X_{j_1}X_{j_2}...X_{j_s}  u|\ \leq\ \int_{Q_X(z_o,2r)}\ |\frac{\partial^k}{\partial t^k} X_{j_1}X_{j_2}...X_{j_s}K(x,\xi;t - \tau)|\ |u(\xi,\tau)|\ d\xi\ d\tau\ .
\end{equation}

In what follows, to simplify the notation we will indicate with $X^s
f$ the derivative $X_{j_1}X_{j_2}...X_{j_s} f$ of a function $f$.
Also, we write $\partial_t f$, instead of $\partial f/\partial t$.
Leibniz rule gives
\[
\partial_t^k\ X^s (f g)\ =\ \sum_{i=0}^k \sum_{l=0}^s \begin{pmatrix} k \\ i \end{pmatrix} \begin{pmatrix} s \\ l \end{pmatrix} \bigg(\partial_t^{k-i} X^{s-l} f\bigg) \bigg(\partial_t^{i} X^{l} g \bigg) \ .
\]

Applying this formula to \eqref{PS8} we find
\begin{align}\label{PS10}
 \partial_t^k\ X^s K\ &=\ \sum_{i=0}^k \sum_{l=0}^s \begin{pmatrix} k \\ i \end{pmatrix} \begin{pmatrix} s \\ l \end{pmatrix} \bigg(\partial_t^{k-i} X^{s-l} p\bigg) \bigg(\partial_t^{i} X^{l} \bigg\{\frac{\partial \zeta}{\partial \tau}\ +\ \sum_{j=1}^m X_j X_j \zeta\ -\ <\overset{\longrightarrow}{div X},X\zeta>\bigg\} \bigg)
\\
& +\ 2\ \sum_{j=1}^m \sum_{i=0}^k \sum_{l=0}^s \begin{pmatrix} k \\ i \end{pmatrix} \begin{pmatrix} s \\ l \end{pmatrix} \bigg(\partial_t^{k-i} X^{s-l} X_jp\bigg) \bigg(\partial_t^{i} X^{l} X_j \zeta\bigg)
\notag\\
& =\ I(x,t;\xi,\tau)\ +\ II(x,t;\xi,\tau)\ .
\notag
\end{align}

Substituting \eqref{PS10}
in \eqref{PS9} we recognize that, in order to complete the proof of the theorem, it will suffice to establish to estimate
\begin{align}\label{PS11}
& \int_{Q_X(z_o,2r)}\ |I(x,t;\xi,\tau)|\ |u(\xi,\tau)|\ d\xi\ d\tau\  +\ \int_{Q_X(z_o,2r)}\ |II(x,t;\xi,\tau)|\ |u(\xi,\tau)|\ d\xi\ d\tau
\\
& \leq\ \frac{C}{r^{2k + s}}\ \frac{1}{|Q_X(z_o,2r)|}\ \int_{Q_X(z_o,2r)}\ |u(\xi,\tau)|\ d\xi\ d\tau\ ,
\notag
\end{align}
for some constant $C = C(X,s,k)>0$. We will prove that, in fact, each of the two terms in \eqref{PS11} is bounded by the quantity in the right-hand side. Also, since these two terms are similar we will only estimate one of them. Keeping in mind \eqref{support} we obtain from Lemma \ref{L:parcutoff}
\begin{align}\label{PS12}
& \int_{Q_X(z_o,2r)}\ |I(x,t;\xi,\tau)|\ |u(\xi,\tau)|\ d\xi\ d\tau
\\
& \leq\ \sum_{i=0}^k \sum_{l=0}^s \begin{pmatrix} k \\ i \end{pmatrix} \begin{pmatrix} s \\ l \end{pmatrix}\ \frac{C(k,i,s,l)}{r^{2i + l + 2}}\ \int_{[Q_X(z_o,2r) \setminus Q_X(z_o,r)] \cap \{\tau < t \}} \bigg|\partial_t^{k-i} X^{s-l} p(x,\xi;t-\tau)\bigg|\ |u(\xi,\tau)|\ d\xi\ d\tau\ .
\notag
\end{align}

We now break the integral in the right-hand side of \eqref{PS12} in two pieces
\begin{align}\label{PS13}
& \int_{[Q_X(z_o,2r) \setminus Q_X(z_o,r)] \cap \{\tau < t \}} \bigg|\partial_t^{k-i} X^{s-l} p(x,\xi;t-\tau)\bigg|\ |u(\xi,\tau)|\ d\xi\ d\tau
\\
& \leq\ \int_{[Q_X(z_o,2r) \setminus Q_X(z_o,r)] \cap \{t_o - r^2 < \tau < t \}} \bigg|\partial_t^{k-i} X^{s-l} p(x,\xi;t-\tau)\bigg|\ |u(\xi,\tau)|\ d\xi\ d\tau
\notag\\
& \int_{[Q_X(z_o,2r) \setminus Q_X(z_o,r)] \cap \{t_o - 4 r^2 < \tau < t_o - r^2 \}} \bigg|\partial_t^{k-i} X^{s-l} p(x,\xi;t-\tau)\bigg|\ |u(\xi,\tau)|\ d\xi\ d\tau
\notag\\
& =\ I'\ +\ I''\ .
\notag
\end{align}

To estimate $I'$ we observe that since $t<t_o$ we can majorize
\begin{align}\label{PS14}
& I'\ \leq\ \int_{[Q_X(z_o,2r) \setminus Q_X(z_o,r)] \cap \{t - r^2 < \tau < t \}} \bigg|\partial_t^{k-i} X^{s-l} p(x,\xi;t-\tau)\bigg|\ |u(\xi,\tau)|\ d\xi\ d\tau
\\
& \leq\ \int_{t - r^2}^t\ \int_{a r < d(x,\xi) < 4 a r} \bigg|\partial_t^{k-i} X^{s-l} p(x,\xi;t-\tau)\bigg|\ |u(\xi,\tau)|\ d\xi\ d\tau\ ,
\notag
\end{align}
where in the second integral we have used \eqref{equiv}.

Theorem \ref{T:KS} implies
\begin{align}\label{PS15}
 \bigg|\partial_t^{k-i} X^{s-l} p(x,\xi;t-\tau)\bigg|\ \leq\ \frac{C(k,i,s,l)}{(t - \tau)^{k - i + \frac{s - l}{2}}}\ \frac{1}{|B(x,\sqrt{t - \tau})|}\ \exp\ \bigg( - \ \frac{M d(x,\xi)^2}{t-\tau}\bigg)\ .
\end{align}

We now observe that when $\tau > t - r^2$ we have from \eqref{dc2}
\begin{equation}\label{PS16}
|B(x,\sqrt{t-\tau})|\ \geq\ C_1\ \bigg(\frac{\sqrt{t-\tau}}{r}\bigg)^Q\ |B(x,r)|\ .
\end{equation}

Inserting \eqref{PS16} into \eqref{PS15} we find that when $\tau > t - r^2$ and $a r < d(x,\xi) < 4 a r$
\begin{align}\label{PS17}
& \bigg|\partial_t^{k-i} X^{s-l} p(x,\xi;t-\tau)\bigg|\ \leq\ \frac{C(k,i,s,l,C_1)}{(t - \tau)^{k - i + \frac{s - l + Q}{2}}}\ \frac{r^Q}{|B(x,r)|}\ \exp\ \bigg( - \ \frac{M a^2 r^2}{t-\tau}\bigg)\ .
\end{align}

Substitution of \eqref{PS17} in \eqref{PS14} leads to the estimate
\begin{align}\label{PS18}
I'\ &\leq\ C(k,i,s,l,C_1)\ \frac{r^Q}{|B(x,r)|}\ \underset{t - r^2 <
\tau <t}{\sup} \bigg\{\frac{1}{(t - \tau)^{k - i + \frac{s - l +
Q}{2}}} \exp\ \bigg( - \ \frac{M a^2 r^2}{t-\tau}\bigg)\bigg\}
\\
& \times \ \int_{Q_X(z_o,2r)} |u(\xi,\tau)|\ d\xi\ d\tau\ .
\notag\\
& \leq\ \frac{r^Q}{|B(x,r)|}\frac{C'(X,k,i,s,l)}{r^{2(k-i) + s - l + Q}}\  \int_{Q_X(z_o,2r)} |u(\xi,\tau)|\ d\xi\ d\tau\ .
\notag\\
& =\ \frac{C'(X,k,i,s,l)}{r^{2(k-i) + s - l}}\ \frac{1}{|B(x,r)|}\ \int_{Q_X(z_o,2r)} |u(\xi,\tau)|\ d\xi\ d\tau\ .
\notag
\end{align}

To estimate $I''$ we observe that, since $(x,t) \in Q_X(z_o,r/2)$, on the region of integration we have $4 r^2 > t - \tau > 3/4 r^2$. Therefore, on such region the following estimate is a direct, and trivial, consequence of \eqref{PS15}
\begin{align}\label{PS119}
 \bigg|\partial_t^{k-i} X^{s-l} p(x,\xi;t-\tau)\bigg|\ \leq\ \frac{C''(k,i,s,l)}{r^{2(k - i) + s - l}}\ \frac{1}{|B(x,\frac{\sqrt{3}}{2} r)|}\ .
\end{align}

By \eqref{dc} we conclude
\begin{equation}\label{PS20}
I''\ \leq\ \frac{C'(X,k,i,s,l)}{r^{2(k-i) + s - l}}\ \frac{1}{|B(x,r)|}\ \int_{Q_X(z_o,2r)} |u(\xi,\tau)|\ d\xi\ d\tau\ .
\end{equation}

Combining \eqref{PS18}, \eqref{PS20} with \eqref{PS13}, and inserting the resulting inequality into \eqref{PS12} we conclude
\begin{equation}\label{PS21}
\int_{Q_X(z_o,2r)}\ |I(x,t;\xi,\tau)|\ |u(\xi,\tau)|\ d\xi\ d\tau\ \leq\ \frac{C(X,s,k)}{r^{2k + s}}\ \frac{1}{r^2 |B_X(x_o,2r)|}\ \int_{Q_X(z_o,2r)} |u(\xi,\tau)|\ d\xi\ d\tau\ .
\end{equation}

By analogous arguments, one obtains the following estimate
\begin{equation}\label{PS22}
\int_{Q_X(z_o,2r)}\ |II(x,t;\xi,\tau)|\ |u(\xi,\tau)|\ d\xi\ d\tau\ \leq\ \frac{C(X,s,k)}{r^{2k + s}}\ \frac{1}{r^2 |B_X(x_o,2r)|}\ \int_{Q_X(z_o,2r)} |u(\xi,\tau)|\ d\xi\ d\tau\ .
\end{equation}

This completes the proof of the theorem.

\end{proof}


\vskip 0.6in















\begin{thebibliography}{99}


\bibitem[Be]{Be}
  A. Bella\"{\i}che, \emph{The tangent space in sub-Riemannian geometry.
Sub-Riemannian geometry}, Progr. Math., \textbf{144}~(1996),
Birkh\"auser, 1-78.

\bibitem[BLU]{BLU}
A. Bonfiglioli, E. Lanconelli \& F. Uguzzoni, \emph{Stratified Lie
groups and potential theory for their sub-Laplacians}, Springer
Monographs in Mathematics. Springer, Berlin, 2007. xxvi+800.


\bibitem[B]{B}
  J. M. Bony, \emph{Principe du maximum, in\'egalit\'e de Harnack et unicit\'e du probl\`eme de Cauchy pour les operateurs elliptique degeneres}, Ann. Inst. Fourier, Grenoble, 1, \textbf{119}~(1969), 277-304.





\bibitem[CDG1]{CDG1}
  L. Capogna, D. Danielli \& N. Garofalo, \emph{Subelliptic mollifiers and a characterization of Rellich and Poincar\'e domains}, Rend. Sem. Mat. Univ. Pol. Torino,  4, \textbf{54}~(1993), 361-386.

\bibitem[CDG2]{CDG2}
\bysame, \emph{The geometric Sobolev embedding for vector fields and
the isoperimetric inequality}, Comm. Anal. and Geom.,
\textbf{2}~(1994), 201-215.



\bibitem[CDG3]{CDG3}
\bysame, \emph{Subelliptic mollifiers and a basic pointwise estimate
of Poincar\'e type}, Math. Zeit., \textbf{226}~(1997), 147-154.


\bibitem[CG]{CG}
L. Capogna \& N. Garofalo, \emph{Boundary behavior of nonegative
solutions of subelliptic equations in NTA domains for
Carnot-Carath\'eodory metrics}, Journal of Fourier Anal. and Appl.,
4 \textbf{4}~(1995).


\bibitem[CGN1]{CGN1}
L. Capogna, N. Garofalo \& D. M. Nhieu, \emph{A version of a theorem
of Dahlberg for the subelliptic Dirichlet problem}, Math. Res.
Letters, \textbf{5} (1998), 541-549.

\bibitem[CGN2]{CGN2}
\bysame, \emph{ Properties of harmonic measures in the Dirichlet
problem for nilpotent Lie groups of Heisenberg type}, Amer. J. Math.
\textbf{124}, vol 2, (2002) 273-306.

\bibitem[CGN3]{CGN3}
\bysame, \emph{Mutual absolute continuity of harmonic and surface
measures for H\"ormander type operators}, Proc. Symposia Pure Math.,
Amer. Math. Soc., volume in honor of V. Maz'ya's 70th birthday, D.
Mitrea, Ed., to appear.





\bibitem[Ch]{Chow}
  W.L. Chow, \emph{\"Uber System von linearen partiellen Differentialgleichungen erster Ordnug}, Math. Ann., \textbf{117}~(1939), 98-105.





\bibitem[CGL]{CGL}
G. Citti, N. Garofalo \& E. Lanconelli, \emph{Harnack's inequality
for sum of squares of vector fields plus a potential}, Amer. J.
Math., 3,\textbf{115}~(1993), 699-734.




\bibitem[E]{E}
L. C. Evans, \emph{Partial Differential Equations}, Grad. Studies in Math., Amer. Mat. Soc. vol.19, 1998.



\bibitem[FSC]{FSC}
  C. Fefferman \& A. Sanchez-Calle, {\em Fundamental solutions for second order subelliptic operators}, Ann. Math., {\bf 124} (1986), 247--272.


\bibitem[F]{F}
G. B. Folland, \emph{Subelliptic estimates and function spaces on
nilpotent Lie groups}, Ark. Math., \textbf{13}~(1975), 161-207.


\bibitem[Fr]{Fr}
A. Friedman, \emph{Partial Differential Equations of Parabolic Type}, Prentice-Hall, Inc., 1964.


\bibitem[G]{G}
N. Garofalo, \emph{Analysis and Geometry of Carnot-Carath\'eodory Spaces, With Applications to Pde's}, Birkh\"auser, book in preparation.



\bibitem[GN1]{GN1}
  N. Garofalo \& D.M. Nhieu,
  \emph{Isoperimetric and Sobolev inequalities for Carnot-Carath\'eodory spaces and the existence of minimal surfaces}, Comm. Pure Appl. Math., \textbf{49}~(1996), 1081-1144.

\bibitem[GN2]{GN2}
  N. Garofalo \& D. M. Nhieu, \emph{Lipschitz continuity, global smooth approximations and extension theorems for Sobolev functions in Carnot-Carath\'eodory spaces}, J. d'Analyse Math., \textbf{74}~(1998), 67-97.


\bibitem[GS]{GS}
N. Garofalo \& F. Segala, \emph{Estimates of the fundamental
solution and Wiener's criterion for the heat equation on the
Heisenberg group}, Indiana Univ. Math. J. \textbf{39}~(1990), no. 4,
1155-1196.


\bibitem[Ga]{Ga}
B. Gaveau, \emph{Principe de moindre action, propagation de la
chaleur et estim\'ees sous elliptiques sur certains groupes
nilpotents}, Acta Math., \textbf{139}~(1977), no. 1-2, 95-153.

\bibitem[H]{H}
  H. H\"ormander, \emph{Hypoelliptic second-order differential equations}, Acta Math., \textbf{119}~(1967), 147-171.



\bibitem[JSC]{JSC}
D. Jerison \& A. S\'anchez-Calle, \emph{Estimates for the heat
kernel for a sum of squares of vector fields}, Indiana Univ. Math.
J., \textbf{35}~(1986), no.4, 835-854.



\bibitem[KS1]{KS1}
S. Kususoka \& D. W. Stroock, \emph{Applications of the Malliavin calculus}, $III$, J. Fac. Sci. Univ. Tokyo, $I$ A, Math., \textbf{38}~(1987), 391-442.

\bibitem[KS2]{KS2}
\bysame, \emph{Long time estimates for the heat kernel associated with a uniformly subelliptic symmetric second order operator}, Annals of Math., \textbf{127}~(1989), 165-189.

\bibitem[LU1]{LU}
E. Lanconelli \& F. Uguzzoni, \emph{On the Poisson kernel for the
Kohn Laplacian}, Rend. Mat. Appl. (7) \textbf{17}~(1997), no. 4,
659--677.

\bibitem[LU2]{LU2}
\bysame, \emph{Degree theory for VMO maps on metric spaces and
applications to Hörmander operators}, Ann. Sc. Norm. Super. Pisa Cl.
Sci. (5) \textbf{1}~(2002), no. 3, 569-601.

\bibitem[NSW]{NSW}
A. Nagel, E.M. Stein \& S. Wainger, \emph{Balls and metrics defined
by vector fields I: basic properties}, Acta Math.
\textbf{155}~(1985), 103-147.

\bibitem[Ra]{Ra}
P. K. Rashevsky, \emph{Any two points of a totally nonholonomic
space may be connected by an admissible line}, Uch. Zap. Ped. Inst.
im. Liebknechta, Ser. Phys. Math., (Russian) \textbf{2}~(1938),
83-94.



\bibitem[RS]{RS}
L. P. Rothschild \&  E. M. Stein,
\emph{Hypoelliptic differential operators and nilpotent groups}.  Acta
Math. \textbf{137}~(1976), 247--320.

\bibitem[SC]{SC}
A. Sanchez-Calle,\emph{Fundamental solutions and geometry of sum of
squares of vector fields}, Inv. Math., \textbf{78}~(1984), 143-160.

\bibitem[S]{S}
E.M. Stein, \emph{Harmonic Analysis: Real Variable Methods,
Orthogonality and Oscillatory Integrals}, Princeton Univ. Press.,
(1993).


\bibitem[VSC]{VSC}
N.  Th. Varopoulos, L. Saloff-Coste \& T. Coulhon, \emph{Analysis and Geometry on Groups}, Cambridge U. Press, 1992.

\end{thebibliography}
\end{document}